\documentclass{amsart}

\usepackage[pdfpagelabels,hidelinks]{hyperref}
\usepackage{xcolor}

\newcommand{\C}{\mathbb{C}}

\newcommand{\Z}{\mathbb{Z}}
\newcommand{\N}{\mathbb{N}}
\newcommand{\F}{\mathbb{F}}
\newcommand{\GL}{\mathbf{GL}}

\DeclareMathOperator{\Tr}{Tr}
\DeclareMathOperator{\supp}{supp}
\DeclareMathOperator{\ext}{Ext}

\newtheorem{theorem}{Theorem}[section]
\newtheorem{lemma}[theorem]{Lemma}
\newtheorem{corollary}[theorem]{Corollary}
\newtheorem{proposition}{Proposition}[section]

\theoremstyle{definition}
\newtheorem{definition}[theorem]{Definition}

\newtheorem{remark}[theorem]{Remark}

\numberwithin{equation}{section}

\begin{document}
	
	\title[Uncertainty principle for small index subgroups]{On an uncertainty principle for small index subgroups of finite fields}
	
	\author[Diego Fernando D\'iaz Padilla]{Diego Fernando Díaz Padilla$^\dagger$}
	\address{Pontificia Universidad Javeriana, Faculty of Science, Department of Mathematics \\ Bogot\'a, Colombia.}
	\curraddr{}
	\email{$^\dagger$ di-diego@javeriana.edu.co\\ $^\ddagger$jesus.ochoa@javeriana.edu.co}
	\thanks{}
	
	\author[Jes\'us Alonso Ochoa Arango]{Jes\'us Alonso Ochoa Arango$^\ddagger$}
	\address{}
	\curraddr{}
	\email{}
	\thanks{}
	
	\date{}
	\dedicatory{}
	
	\begin{abstract}
		In this paper we continue the study of the \textit{nonvanishing minors property} (NVM) initiated by Garcia, Karaali and Katz, for the compressed Fourier matrix attached to a subgroup $H$ of the multiplicative group of a finite field $\mathbb{F}_q$ and a character $\chi$ defined over $H$.\@ Here we provide a characterization of this aforementioned property for \textit{symmetries} arising from an index-3 subgroup $H$ and a nontrivial character $\chi$.
	\end{abstract}
	
	\maketitle
	
	\section{Introduction}
	
	In discrete Fourier analysis, \textit{uncertainty principles} have played an essential role due to their profound implications in signal processing.\@ The study of these relations began in 1989 with the well-known theorem of D.\@ Donoho and P.\@ Stark in \cite{DonohoStark}.\@ Before stating this result, let us remind that if $G$ is a finite group, the group algebra of $G$ over $\C$, denoted by $\C[G]$, is the $\C$-vector space spanned by $G$,
	$$
	\C[G]= \biggl\{ \sum_{g\in G} f_{g}g\colon f_{g}\in \C\biggl\},
	$$
	endowed with the convolution product.\@ In what follows, $G$ will denote an arbitrary finite abelian group.\@ Recall that given $f\in \C[G]$, the \textit{Fourier transform} of $f$ is the function $\hat{f}\colon \hat{G}\to \C$ given by
	$$
	\hat{f}(\chi):= \sum_{g\in G}f_g\chi(g),
	$$
	where $\hat{G}$ denotes the group of characters of $G$.\@ The Donoho-Stark uncertainty principle for finite abelian groups states that if $f\in \C[G]$ is nonzero, then:
	$$
	|\supp(f)||\supp(\hat{f})|\geq |G|,
	$$
	where $\supp(f):=\{g\in G\colon f_g\neq 0\}$ and $|X|$ denotes the cardinality of a set $X$.\@ Various generalizations and results emerged from this principle, for instance \cite{MESHULAM1992401,KennanSmith90}, but perhaps the most important of all these is due to T.\@ Tao, who in \cite{Tao2003} proved that by considering $G$ to be the cyclic group $\Z/p\Z$ of prime order $p$, a substantial improvement can be obtained: if $f\in \C[G]$ is nonzero, then: 
	\begin{align}\label{eq:Tao}
		|\supp(f)|+|\supp(\hat{f})|\geq p+1.
	\end{align}
	
	This remarkable result, also discovered independently by A.\@ Biró \cite{biro_schweitzer_1998} and R.\@ Meshulam \cite{MESHULAM200663}, led to the developments in \cite{TaoCandesSignal} that gave rise to the field of \textit{compressed sensing} and several new uncertainty relations.\@ For more about this see \cite{MURTY2012214}, where the uncertainty result was generalized to arbitrary finite cyclic groups, or \cite{BORELLO2022112670,evra2017good,BorelloWillemsZini} for relations to the performance of cyclic codes and group codes; other studies can be found in \cite{Bonami2013,GHOBBER2011751,NICOLA2023109924}.
	
	At the core of these improvements is \textit{Chebotar\"ev's theorem} on roots of unity, originally proposed by A.\@ Ostrowski and proved by Chebotar\"ev in 1926 \cite{Stevenhagen1996ChebotarvAH}.\@ In Tao's paper \cite{Tao2003} it is proved that, indeed, Chebotar\"ev's theorem is equivalent to \eqref{eq:Tao}.\@ The theorem establishes that every minor of the \textit{discrete Fourier transform matrix} (DFT matrix) is nonzero if the matrix has prime order.\@ To be accurate, 
	\begin{theorem}[Chebotar\"ev]\label{thm:Chebotarev}
		Let $p$ be a prime and $\zeta$ a primitive $p$-th root of unity.\@ For every pair of subsets $I,J\subseteq \F_{p}$ with the same cardinality, the matrix $(\zeta^{ij})_{i\in I,j\in J}$ is nonsingular, that is, it has nonvanishing determinant.
	\end{theorem}
	
	The property that every minor of a given matrix is nonzero is of particular interest in this paper, so we introduce the following definition:
	\begin{definition}[Nonvanishing minors property]
	A matrix $A=(a_{i,j})_{1\leq i,j\leq n}$ with complex entries is said to have the \textit{nonvanishing minors property} (NVM) if for every $I,J\subseteq \{1,\dots,n\}$ with $|I|=|J|$, the determinant of $(a_{i,j})_{i\in I,j\in J}$ is nonzero.
	\end{definition}
	
	The equivalence of \eqref{eq:Tao} with Chebotar\"ev's theorem raises the question of whether other transformations related to the discrete Fourier transform exhibit the nonvanishing minors property in their matrix representations, and if this leads to improved uncertainty principles.\@ For example, if $n\geq 1$ is an odd integer then it can be proved that the $\frac{n+1}{2}\times \frac{n+1}{2}$ matrix attached to the \textit{discrete cosine transform} (DCT) satisfies the NVM property if and only if $n$ is prime or $n=1$; similarly, if we let $n\geq 3$ be an odd integer, in the case of the \textit{discrete sine transform} (DST) it can be proved that the $\frac{n-1}{2}\times \frac{n-1}{2}$ matrix attached to this transform satisfies the NVM property if and only if $n$ is a prime; see \cite{GARCIA2021899} for more details.
	
	Let $\F_{q}$ denote the finite field with $q$ elements. In \cite{GARCIA2021899} S.\@ R.\@ Garcia, G.\@ Karaali, and D.\@ Katz made significant improvements on \eqref{eq:Tao} by introducing a general notion of symmetry on elements of $\C[\F_{q}]$ that encompasses the aforementioned DCT and DST cases.\@ Given a subgroup $H \leq \mathbb{F}_q^{\times}$ and a complex character $\chi:H \to \mathbb{C}^{\times}$, an element $f=\sum f_{a}a\in \C[\F_{q}]$ is said to be \textit{$\chi$-symmetric} if $f_{ha} = \chi(h)f_{a}$ for all $h\in H$ and $a\in \F_{q}$.\@ When considering the Fourier transform on $\C[\F_{q}]$ restricted to the subspace of $\chi$-symmetric elements we arrive at the \textit{Compressed Fourier Transform} (CFT) attached to the pair $(H, \chi)$; see Definition\@ \ref{Def-CFT}.\@ For instance, if $p$ is an odd prime, $H=\{-1,1\}$ and the character $\chi$ is such that $\chi(-1)=-1$, then $\chi$-symmetric elements correspond precisely to elements $f$ such that $f_{-a}=-f_{a}$ and the CFT corresponds to the DST.\@ The introduction of the Compressed Fourier Transform led to the study of the NVM property for its associated matrix.\@ For non-prime finite fields, general conditions for the NVM property to be satisfied for the CFT matrix were not obtained, however for certain subgroups $H$ of a non-prime field $\F_{q}$ they arrive, for example, to the following results:
	\begin{itemize}
		\item If $H=\{1\}$ then the CFT matrix does not satisfy the nonvanishing minors property; refer to \cite[Corollary 6.2]{GARCIA2021899}.
		\item If $H=\F_{q}^{\times}$ or, in the case $q$ is odd, if $H$ an index-$2$ subgroup, and $\chi$ is the trivial character, then the CFT matrix exhibits the nonvanishing minors property; for more details see \cite[Proposition 6.5 and Theorem 6.6]{GARCIA2021899}.
		\item If $q$ is odd, $H$ an index-$2$ subgroup and $\chi$ nontrivial, a characterization was found in terms of Gaussian sums of character extensions \cite[Theorem 6.7]{GARCIA2021899}. 
		\item Again, if $q$ is odd, $3\mid (q-1)$, $H$ is an index-$3$ subgroup and $\chi$ is the trivial character, the NVM property holds if and only if $p\equiv 1$ (mod 3), where $p$ is the characteristic of $\F_{q}$; see \cite[Theorem 6.12]{GARCIA2021899}.
	\end{itemize}
	
	In this paper, we pursue this approach by providing concise necessary and sufficient conditions for the NVM property to hold in the case of index-$3$ subgroups $H$ and nontrivial characters $\chi$.
	\subsection{Structure of the paper}
	In Section \ref{sec:preliminaries} we will review some basic notions of character theory and discrete Fourier analysis, and then introduce the necessary ideas from \cite{GARCIA2021899}, such as $\chi$-symmetry and the compressed Fourier transform.\@ In Section \ref{sec:ourresult} we present our main result, Theorem \ref{thm:MainResult}, which characterizes the nonvanishing minors property of the CFT matrix for index-$3$ subgroups and nontrivial characters.
	
	\section{Preliminaries}\label{sec:preliminaries}
	
	\subsection{Characters and the Fourier transform}
	We begin by recalling the basic concepts of character theory on finite fields.\@ For a more detailed explanation we refer the reader to \cite{lidl1997finite}. 
	
	An \textit{additive character} of $\F_{q}$ is a group homomorphism from the additive group of $\F_{q}$ into the group $\C^{\times}$.\@ Similarly, a \textit{multiplicative character} of $\F_{q}$ is a group homomorphism now defined on the multiplicative group $\F_{q}^{\times} = \F_{q}-\{0\}$.\@ It is well known that one way to get a complete description of additive characters is by introducing the \textit{canonical additive character}:\@ let $p$ be the characteristic of $\F_{q}$, so that $q=p^{m}$ for some $m\in \N$, and consider the additive character $\varepsilon\colon \F_{q} \to \C^{\times}$ defined by $\varepsilon(x) := e^{2\pi i \Tr(x)/p}$ for all $x\in \F_{q}$, where 
	$$
	\Tr(x) := x + x^{p} + \cdots + x^{p^{m-1}},% \;\; \text{ for all }x\in \F_{q},
	$$ 
	is the \textit{absolute trace} map from $\F_{q}$ to $\F_{p}$.\@ It can be shown that for every additive character $\psi$ there exists $a\in \F_{q}$ such that $\psi(x) = \varepsilon(ax)$ for all $x\in \F_{q}$, which allows us to define the character $\varepsilon_{a}\colon \F_{q} \to \C^{\times}$ given by $\varepsilon_{a}(x)=\varepsilon(ax)$ for all $x\in \F_{q}$.\@ Denote by $\hat{\F}_{q}$ the group of additive characters of $\F_{q}$, and if $S\subseteq \F_{q}$ define $\varepsilon_{S}:=\{\varepsilon_{s}\colon s\in S\} \subseteq\hat{\F}_{q}$, as in \cite{GARCIA2021899}, so that $\varepsilon_{\F_{q}}=\hat{\F}_{q}$.
	
	There is a relevant connection between multiplicative and additive characters in a finite field in terms of certain exponential sums called Gaussian sums.\@ Let $\chi$ be a multiplicative and $\psi$ an additive character of $\F_{q}$.\@ The \textit{Gaussian sum} $G(\chi,\psi)$ is defined as
	$$
	G(\chi,\psi):= \sum_{c\in \F_{q}^{\times}}\chi(c)\psi(c),
	$$
	and we will use the notation $G(\chi)$ when $\psi = \varepsilon$.\@ Perhaps one of the most important facts about Gaussian sums, and one that we will use later, is that if $\psi$ and $\chi$ are both nontrivial, then we have $|G(\chi,\psi)| = \sqrt{q}$; see \cite[Theorem 5.11]{lidl1997finite} for more details.\@ The sum $G(\chi,\psi)$ is closely related to the Fourier expansion of the multiplicative character $\chi$, as we now show.\@ Let $\C^{\hat{\F}_{q}}$ be the $\C$-vector space of functions from $\hat{\F}_{q}$ to $\C$ (the expression $X^{Y}$ is interpreted similarly) endowed with pointwise multiplication, and define the Fourier transform of $f\in \C[\F_{q}]$ as the map $\hat{f}\colon \hat{\F}_{q}\to \C$ given by:
	$$
	\hat{f}(\psi):=\sum_{a\in \F_{q}} f_{a}\psi(a). %\;\; \text{ for all }\psi\in \hat{\F}_{q}.
	$$
	The $\C$-algebra isomorphism $\mathcal{F}\colon \C[\F_{q}] \to \C^{\hat{\F}_{q}}$ given by $\mathcal{F}(f):= \hat{f}$ is called the \textit{Fourier transform on $\C[\F_{q}]$}, and its inverse $\mathcal{F}^{-1}$ is given by $\hat{f} \mapsto \mathcal{F}^{-1}(\hat{f})=\sum f_a a$, where
	$$
	f_{a} = \tfrac{1}{q}\sum_{\psi\in \hat{\F}_{q}} \overline{\psi(a)}\hat{f}(\psi).
	$$
	We can extend a multiplicative character $\chi\colon \F_{q}^{\times}\to \C^{\times}$ to a multiplicative map defined over the whole $\F_{q}$ by simply mapping $\chi(0)=0$.\@ If we use the standard identification $\C[\F_{q}]\cong \C^{\F_{q}}$ and the definition of the Fourier transform we can prove that $\hat{\chi}(\overline{\psi}) = G(\chi,\overline{\psi})$ for every $\psi\in \hat{\F}_{q}$.\@ Moreover, if we seek the value of $\chi$ at $c\in \F_{q}^{\times}$, we can use the Fourier inversion formula to get the remarkable expression:
	$$
	\chi(c) = \tfrac{1}{q}\sum_{\psi\in \hat{\F}_{q}}G(\chi,\overline{\psi})\psi(c),
	$$
	in which Gaussian sums are precisely the Fourier coefficients in this expansion.
	
	\subsection{Compressed Fourier transform}
	We now introduce, with some small modifications, the main definitions from \cite{GARCIA2021899}.\@ Let $H$ be a subgroup of the multiplicative group $\F_q^\times$ and $\chi\colon H\to \C^\times$ a character.\@ Let $\GL(V)$ denote the group of automorphisms on a $\C$-vector space $V$, and define the map $\mathcal{L}(\chi)\colon H\to \GL(\C[\F_q])$ as follows: 
	$$
	\mathcal{L}(\chi)_{h} \biggl( \sum_{a\in \F_{q}}f_{a} a \biggl)  := \sum_{a\in \F_{q}}\chi(h) f_{a} ha.
	$$
	The map $\mathcal{L}(\chi)$ is a group homomorphism and $\mathcal{L}(\chi)_{h}$ is a linear isomorphism of $\C$-vector spaces for each $h\in H$.\@ Intuitively, $\mathcal{L}(\chi)_{h}$ permutes each coefficient of $f\in \C[\F_{q}]$ and scales them by a root of unity.
	
	We are interested in the elements of $\C[\F_{q}]$ that are invariant under the action of $\mathcal{L}(\chi)$, that is, elements of the set:
	$$
	\C[\F_{q}]^{\chi}:= \biggl\{ f\in \C[\F_{q}]\colon \mathcal{L}(\chi)_{h}(f) = f, \text{ for all }h\in H \biggl\} .
	$$
	It can be easily shown that the set $\C[\F_{q}]^{\chi}$ is a $\C$-vector subspace of $\C[\F_{q}]$.\@ This subspace is actually $H$-invariant, that is to say, $\mathcal{L}(\chi)_{h}(f)\in \C[\F_{q}]^{\chi}$ for all $f\in \C[\F_{q}]^{\chi}$ and $h\in H$.\@ The dependency on both the subgroup $H$ and the character $\chi$ leads to the following definition:
	
	\begin{definition}[$\chi$-symmetry]
		Let $H$ be a subgroup of $\F_q^\times$ and $\chi\colon H \to \C^{\times}$ be a character.\@ Elements of $\C[\F_{q}]^{\chi}$ are called \textit{$\chi$-symmetric}, or equivalently, $f\in \C[\F_{q}]$ is said to be \textit{$\chi$-symmetric} provided that $f_{ha} = \chi(h)f_{a}$ for all $h\in H$ and $a\in \F_{q}$.
	\end{definition}
	
	Let us recall that, given $H$ a subgroup of $\F_{q}^{\times}$, the $H$-orbits of $\F_{q}$ are of the form $Ha=\{ha\colon h\in H\}$ for $a\in \F_{q}$, and when $a\neq 0$ they correspond precisely to the cosets of $H$ in the group $\F_{q}^{\times}$.\@ We say that $(\chi,S)$ is an \textit{orbit-representative pair of} $H$ if $S$ is a complete set of representatives of the $H$-orbits of $\F_{q}$ if $\chi$ is trivial, or of all of $\F_{q}^{\times}$ if $\chi$ is nontrivial.\@ If additionally, we have another set $R$ with the same property, then $(\chi,R,S)$ is called an \textit{orbit-representative 3-tuple of} $H$. 
	
	\begin{definition}[Compressed Fourier transform]\label{Def-CFT}
		Let $H$ be a subgroup of $\F_{q}^{\times}$ and $\chi\colon H\to \C^{\times}$ be a character.\@ Let $(\chi,S)$ be an orbit-representative pair of $H$.\@ Recall that $\varepsilon_{S}$ denotes the set of additive characters of the form $\varepsilon_{s}$ for $s\in S$.\@ The $\C$-vector space isomorphism
		\begin{align*}
			\mathcal{F}_{\chi}\colon \C[\F_{q}]^{\chi} &\to \C^{\varepsilon_{S}}\\
			f &\mapsto \hat{f}|_{\varepsilon_{S}}
		\end{align*}
		is referred to as the $(\chi,S)$-\textit{compressed Fourier transform}.
	\end{definition}
	
	\begin{remark}
		The fact that $\mathcal{F}_{\chi}$ is an isomorphism \cite[Proposition 3.10]{GARCIA2021899} shows that a $\chi$-symmetric element $f$ can be reconstructed with exactly $[\F_{q}^{\times}:H]$ measurements of its Fourier transform when $\chi$ is nontrivial, and with $[\F_{q}^{\times}:H] + 1$ measurements when $\chi$ is trivial, that is, one measurement on each orbit is sufficient to achieve this by the invertibility of $\mathcal{F}_{\chi}$.
	\end{remark}
	
	To obtain a matrix representation for the compressed Fourier transform, it is necessary to determine some basis for $\C[\F_q]^{\chi}$.\@ To this end, in \cite[Lemma 3.9]{GARCIA2021899} the authors attach a suitable basis $\{u_{\chi, r}\}_{r \in R}$ to each orbit-representative pair $(\chi,R)$ of $H$.\@ Thus, if we fix an orbit-representative 3-tuple $(\chi,R,S)$ of $H$, where $R$ and $S$ are endowed with some orderings, the representation matrix in this basis of the Compressed Fourier Transform (CFT) is referred to as the \textit{$(\chi,R,S)$-compressed Fourier matrix}.\@ For our purpose, it will not be necessary to introduce this basis since, as will be seen in the next section, an explicit expression for the entries of the $(\chi,R,S)$-CFT matrix is already known \cite{GARCIA2021899}.\@ Note also that the order of this matrix is $[\F_{q}^{\times}:H]$ when $\chi$ is nontrivial, and $[\F_{q}^{\times}:H] + 1$ if $\chi$ is trivial .\@ If $S=R$ and we impose the same ordering, then the $(\chi,R,S)$-compressed Fourier matrix is symmetric.
	
	Since the NVM property for CFT matrices is independent of the choice of sets of representatives and orderings then, for simplicity, it is said that the pair $(\F_{q},\chi)$ has or does not have the \textit{nonvanishing minors property}.\@ The next proposition, proved in \cite{GARCIA2021899}, provides a criteria, in terms of $\chi$-symmetric functions, for a pair $(\F_{q},\chi)$ to have the NVM property:
	\begin{proposition}[{\cite[Proposition 4.8]{GARCIA2021899}}]\label{thm:SUP}
		Let $H\leq \F_{q}^{\times}$ and $\chi\colon H \to \C^{\times}$ be a character.\@ Then $(\F_{q},\chi)$ has the nonvanishing minors property if and only if for every nonzero $\chi$-symmetric element $f\in \C[\F_{q}]^{\chi}$ we have,
		\begin{enumerate}%[\normalfont($i$)]
			\item if $\chi$ is nontrivial,
			$$
			|\supp(f)|+ |\supp(\hat{f})|\geq q+|H|-1,
			$$
			\item if $\chi$ is trivial,
			$$
			|\supp(f)|+ |\supp(\hat{f})|\geq\begin{cases}
				q+2|H|-1	& \text{if $f_{0} = 0$ and $\hat{f}(\epsilon_{0}) = 0$},  \\
				q+|H|	& \text{if $f_{0}=0$ or $\hat{f}(\epsilon_{0}) = 0$}, \\
				q+1	& \text{otherwise}. 
			\end{cases}
			$$
		\end{enumerate}
	\end{proposition}
	The above proposition gives us an alternative version of the NVM property directly related to the uncertainty principle of Biró-Meshulam-Tao. 
	
	\section{Index-3 subgroups and nontrivial characters}\label{sec:ourresult}
	The nonvanishing minors property of $(\F_{q},\chi)$ when $H$ is an index-3 subgroup and $\chi$ is the trivial character is satisfied if and only if $p\equiv 1$ (mod $3$), where $p$ is the characteristic of $\F_{q}$ \cite[Theorem 6.12]{GARCIA2021899}.\@ Our result completes the characterization for nontrivial characters by giving concise necessary and sufficient conditions for the NVM property to hold.
	
	We shall comment on character extensions.\@ Suppose $H$ is a subgroup of a finite abelian group $G$, and let $\chi\colon H\to \C^{\times}$ be a character.\@ If we denote the set of extensions of $\chi$ to $G$ by $\ext(\chi)$, it can be proved that its cardinality is the index $[G:H]$.\@ To describe this set of extensions, consider first the \textit{annihilator} of $H$ in $\hat{G}$:
	$$
	\text{Ann}(H):=\left\lbrace \chi\in \hat{G}\colon \chi(h)=1, \; \text{for all }h\in H  \right\rbrace.
	$$
	It can be shown that the annihilator of $H$ in $\hat{G}$ is a subgroup of $\hat{G}$ of order $[G:H]$.\@ If we write $s = [G:H]$ and $\text{Ann}(H) = \{\vartheta_{0},\vartheta_{1},\dots,\vartheta_{s-1}\}$ then, given an extension $\varphi_{0}$ of $\chi$, it is clear that for every $\vartheta_{i}\in \text{Ann}(H)$ the product $\varphi_{0}\vartheta_{i}$ is an extension of $\chi$, and there are precisely $s$ extensions, hence:
	$$
	\ext(\chi) = \varphi_{0}\text{Ann}(H) = \{\varphi_{0}\vartheta_{i}\colon i= 0,\dots,s-1\}.
	$$
	
	Given $H\leq \F_{q}^{\times}$ and a character $\chi\colon H\to \C^{\times}$, the Gaussian sums (with $\psi = \varepsilon$) of the $s=[\F_{q}^{\times}:H]$ character extensions $\varphi_{0},\varphi_{1},\dots,\varphi_{s-1}$ are denoted as $G_{i}$ for all $i\in\{0,1,\dots,s-1\}$.\@ Now we are ready to state the following technical lemma that provides the entries of the compressed Fourier matrices:
	
	\begin{lemma}[{\cite[Lemma 6.4]{GARCIA2021899}}]\label{thm:GaussMatrix}
		Let $\F_{q}$ be a finite field, let $m$ be a positive integer such that $m\mid (q-1)$, and $H$ be the unique index-$m$ subgroup of $\F_{q}^{\times}$.\@ Let $\chi\colon H\to \C^{\times}$ be a character and $(\chi,R,S)$ an orbit-representative 3-tuple of $H$.\@ Then for any $r\in R$ and $s\in S$, the $(r,s)$-entry of a $(\chi,R,S)$-compressed Fourier matrix is
		$$
		[\mathcal{F}_{\chi}]_{r,s} = \begin{cases}
			|H|, & \text{if $rs=0$},\\
			\displaystyle\frac{1}{m}\sum_{i=0}^{m-1}\overline{\varphi_{i}}(rs) G_{i}, & \text{if $rs\neq 0$}.
		\end{cases}
		$$
	\end{lemma}
	
	When fixing a character $\chi\colon H\to \C^{\times}$ of an index-3 subgroup $H\leq \F_{q}^{\times}$, given the notation $G_{i}$ for the Gaussian sums of its character extensions, we introduce for simplicity the notation $T_{j}:= \sum_{i=0}^{2}\zeta_{3}^{ji}G_{i}$ for $j\in \Z$, where $\zeta_{3}=e^{2\pi i /3}$.
	
	\begin{theorem}\label{thm:MainResult}
		Let $\F_{q}$ be a finite field such that $3\mid(q-1)$, let $H$ be the unique index-3 subgroup in $\F_{q}^{\times}$, and let $\chi\colon H\to \C^{\times}$ be a nontrivial character.\@ Then the pair $(\F_{q},\chi)$ has the nonvanishing minors property if and only if $G_{i}\neq G_{j}$ for some $i,j\in\{1,2,3\}$ and $T_{0}\neq 0$.
	\end{theorem}
	\begin{proof}
		Let $\kappa$ and $\overline{\kappa}$ be so that $\{\chi_{0},\kappa,\overline{\kappa}\}$ is the annihilator of $H$, that is, $\kappa$ and $\overline{\kappa}$ are the cubic multiplicative characters of $\F_{q}^{\times}$.\@ Let $\alpha\in \F_{q}$ be such that $\overline{\kappa}(\alpha) = \zeta_{3}$.\@ Consider $R=S=\{1,\alpha,\alpha^{2}\}$; this way $(\chi,R,S)$ is an orbit-representative 3-tuple of $H$.\@ Let $\varphi_{0}$ be a character extension of $\chi$ and $\varphi_{1},\varphi_{2}$ the other two extensions $\varphi_{0}\kappa$ and $\varphi_{0}\overline{\kappa}$ respectively.\@ The $(\chi,R,S)$-compressed Fourier matrix is then:
		$$
		\begin{bmatrix}
			\frac{T_{0}}{3} & \tfrac{\overline{\varphi_{0}}(\alpha) T_{1}}{3} &  \frac{\overline{\varphi_{0}}(\alpha^{2}) T_{2}}{3} \\
			\frac{\overline{\varphi_{0}}(\alpha) T_{1}}{3} & \frac{\overline{\varphi_{0}}(\alpha^{2}) T_{2}}{3} & \frac{T_{0}}{3}\\
			\frac{\overline{\varphi_{0}}(\alpha^{2}) T_{2}}{3} & \frac{T_{0}}{3}  & \frac{\overline{\varphi_{0}}(\alpha) T_{1}}{3}
		\end{bmatrix}.
		$$
		We may scale rows and columns to obtain the matrix
		$$
		M = \begin{bmatrix}
			T_{0} & T_{1} &   T_{2}\\
			T_{1} & T_{2} & T_{0}\\
			T_{2} & T_{0} & T_{1}
		\end{bmatrix},
		$$
		which has the nonvanishing minors property if and only if the $(\chi,R,S)$-compressed Fourier matrix does.\@ The minors of $1\times 1$ submatrices are precisely the entries $T_{0}$, $T_{1}$ and $T_{2}$.\@ For the minors of $2\times 2$ submatrices one can check that these are, up to sign, of the form $T_{i+1}T_{i+2}-T_{i}^{2}$ for $i\in \{0,1,2\}$, where the index $j$ in $T_{j}$ is considered mod $3$.\@ The result of this expression can be reduced by grouping the products of Gaussian sums and using the fact that $\zeta_{3}^{2}+\zeta_{3} - 2 = -3$:
		\begin{align*}
			T_{i+1}T_{i+2}-T_{i}^{2} &= -3\left( \zeta_{3}^{2i}G_{0}G_{2} + \zeta_{3}^{i}G_{0}G_{1} + G_{1}G_{2} \right)\\
			&= -3G_{0}G_{1}G_{2}\left( \frac{\zeta_{3}^{2i}}{G_{1}} + \frac{\zeta_{3}^{i}}{G_{2}} + \frac{1}{G_{0}} \right) \\
			&= -3G_{0}G_{1}G_{2}\left( \frac{\zeta_{3}^{2i}\overline{G_{1}}}{|G_{1}|^{2}} + \frac{\zeta_{3}^{i}\overline{G_2}}{|G_{2}|^{2}} + \frac{\overline{G_{0}}}{|G_{0}|^{2}} \right).
		\end{align*}
		Since Gaussian sums all have absolute value $\sqrt{q}$ we get
		\begin{align*}
			T_{i+1}T_{i+2}-T_{i}^{2} &= -\frac{3G_{0}G_{1}G_{2}}{q}\left( \overline{\zeta_{3}^{i}G_{1} + \zeta_{3}^{2i}G_{2} + G_{0}} \right)\\
			&=  -\frac{3G_{0}G_{1}G_{2}}{q} \overline{T_{i}}.
		\end{align*}
		Therefore the $2\times 2$ minors can be reduced to the entries $T_{j}$ for $j\in \{0,1,2\}$, so that $T_{i+1}T_{i+2}-T_{i}^{2}=0$ if and only if $T_{i}=0$.\@ Finally, the determinant of $M$, which is the only minor of a $3\times 3$ submatrix, is $-27G_{0}G_{1}G_{2}$ and is never zero.\@ With these results at hand, the nonvanishing minors property is satisfied if and only if $T_{0}$, $T_{1}$ and $T_{2}$ are all nonzero.\@ First, suppose the nonvanishing minors property holds, then we just have to show that $G_{i}\neq G_{j}$ for some $i,j\in\{0,1,2\}$.\@ If $G_{0}=G_{1}=G_{2}$, then $T_{1}=T_{2}=0$ arriving at a contradiction. 
		
		For the converse, suppose $T_{0}$ is nonzero and $G_{i}\neq G_{j}$ for some $i,j\in\{0,1,2\}$.\@ Assume $T_{1}=0$, then since Gaussian sums have absolute value $\sqrt{q}$ it follows that $\zeta_{3} G_{1} = \zeta_{3}^{\gamma}G_{0}$ and $\zeta_{3}^{2}G_{2} = \zeta_{3}^{\beta}G_{0}$, for some combination $\gamma,\beta\in \{1,2\}$ such that $\gamma + \beta = 3$.\@ If $\gamma=1$ and $\beta=2$, then $G_{1} = G_{0}= G_{2}$, which is not possible, and if $\gamma=2$ and $\beta=1$, then $G_{1}=\zeta_{3} G_{0}$ and $G_{2} = \zeta_{3}^2 G_{0}$ which leads to $T_{0}=0$, again a contradiction.\@ Following, assume $T_{2}=0$, then again $\zeta_{3}^{2} G_{1} = \zeta_{3}^{\gamma}G_{0}$ and $\zeta_{3} G_{2} = \zeta_{3}^{\beta}G_{0}$ for combination $\gamma,\beta\in \{1,2\}$ such that $\gamma + \beta = 3$.\@ If $\gamma=1$ and $\beta=2$, then $G_{1}= \zeta_{3}^{2}G_{0}$ and $G_{2} = \zeta_{3} G_{0}$ which leads to $T_{0}=0$, and if $\gamma=2$ and $\beta=1$, then $G_{1} = G_{0}=G_{2}$.\@ Thus, the result holds.
	\end{proof}
	
	An equivalent formulation of Theorem \ref{thm:MainResult} in terms of an uncertainty principle can be achieved with Proposition \ref{thm:SUP}:
	\begin{corollary}
		Let $\F_{q}$ be a finite field with $3\mid(q-1)$, let $H$ be the unique index-3 subgroup of $\F_{q}^{\times}$.\@ Let $\chi\colon H\to \C^{\times}$ be a nontrivial character.\@ For every nonzero $\chi$-symmetric element $f\in \C[\F_{q}]^{\chi}$ we have 
		$$|\supp(f)|+ |\supp(\hat{f})|\geq q+\tfrac{q-1}{3}-1,$$ 
		if and only if $G_{i}\neq G_{j}$ for some $i,j\in\{1,2,3\}$ and $T_{0}\neq 0$.
	\end{corollary}
	\begin{proof}
		It is a direct consequence of Theorem \ref{thm:MainResult} and Proposition \ref{thm:SUP}.
	\end{proof}
	
	\begin{remark}
		Before finishing it is worth mentioning a few words about symmetric elements in the complex group algebra of a finite field $\mathbb{F}_q$.\@ In the trivial character case, the symmetric elements boil down to those elements $f\in \C[\F_{q}]$ with constant value $f_{a}$ on each $H$-orbit $Ha$, for all $a\in \F_{q}$.\@ When $\chi$ is non trivial, the $\chi$-symmetric elements can be described as follows: suppose $d\;|\;(q-1)$ and that $H$ is the unique subgroup of order $d$, so that $H=\left\langle \omega \right\rangle $ with $\omega$ a primitive $d^{\text{th}}$ root of unity in $\F_{q}^{\times}$.\@ Let $\chi\colon H \to \C^{\times}$ be the character defined by $\chi(\omega) = \zeta_{d}$, where $\zeta_{d}=e^{2\pi i /d}$.\@ All other characters are of the form $\varphi=\chi^{k}$ for some $k\in \{0,1,\dots,d-1\}$, consequently $\varphi(\omega) = \zeta_{d}^{k}$.\@ Then an element $f\in \C[\F_{q}]$ is $\varphi$-symmetric if and only if $f_{\omega^{j}a} = \zeta_{d}^{kj} f_{a}$ for all $j\in\{0,\dots,d-1\}$.
	\end{remark} 
	With regard to the NVM property for compressed Fourier matrices, the question remains whether more characterizations can be found for subgroups of larger index in terms of concise conditions, both for trivial and nontrivial characters.
	
	\section{Acknowledgments}
	The authors would like to thank Daniel J. Katz for his suggestions and observations regarding Theorem \ref{thm:MainResult} that helped to improve the result.\@ This work was supported by Pontificia Universidad Javeriana at Bogotá, Colombia, under the research project with ID 9714.

\end{document}